\documentclass[11pt,reqno]{amsart}

\usepackage{mathrsfs,amsmath,amsthm,amssymb,thmtools,tikz-cd,comment,hyperref,cleveref,enumitem,breqn,float,caption,subfig,ytableau,tikz,cases,multirow,mathtools}
\usepackage{xcolor}
\usetikzlibrary{positioning, arrows.meta, calc}

\usetikzlibrary{decorations.pathreplacing,
	calligraphy,
	matrix}

\allowdisplaybreaks
\theoremstyle{plain}
\newtheorem{theorem}{Theorem}[section]
\newtheorem{lemma}[theorem]{Lemma}
\newtheorem{proposition}[theorem]{Proposition}

\newtheorem{thmalphabetintro}{Theorem}

\newtheorem{thmalphabetmaintext}{Theorem}

\theoremstyle{definition}

\newtheorem{question}[theorem]{Question}

\newtheorem{open problem}[theorem]{Open Problem}
\newtheorem{remark and notation}[theorem]{Remark and Notation}
\newtheorem{remark and definition}[theorem]{Remark and Definition}
\newtheorem{definition and notation}[theorem]{Definition and Notation}
\newtheorem{notation and convention}[theorem]{Notation and convention}
\newtheorem{convention and notation}[theorem]{Convention and notation}

\def \p {\mathbb{P}}
\def \a {\mathbb{A}}

\def \c {\mathbb{C}}

\def \q {\mathbb{Q}}

\def \Im {\operatorname{Im}}

\def \GL {\operatorname{GL}}
\def \rank {\operatorname{rank}}

\def \det {\operatorname{det}}
\def \per {\operatorname{perm}}
\def \rk {\operatorname{\mathbf{R}}}
\def \wrk {\operatorname{\mathbf{WR}}}

\def \brk {\operatorname{\underline{\mathbf{R}}}}

\def \sgn {\operatorname{sgn}}
\def \hom {\operatorname{Hom}}

\def \Seg {\operatorname{Seg}}
\def \Mat {\operatorname{Mat}}
\def \Sym {\operatorname{Sym}}
\def \Gr {\operatorname{Gr}}

\def \la{\langle}
\def \ra{\rangle}

\def \vmd[#1]{\text{VMD$^{#1}$}}
\def \dpv[#1]{\text{dP$^{#1}$}}
\def \acm[#1,#2]{\text{ACMD$^{#1}_{#2}$}}
\def \propp[#1,#2]{\text{$P_{{#1},{#2}}$}}
\def \propa[#1,#2]{\text{$A_{{#1},{#2}}$}}

\def \blue[#1]{\textcolor{blue}{#1}}

\newcommand{\W}{\tikz{\useasboundingbox (0,0) rectangle (0.5,0.5);
                      \draw[fill=white]   (0,0) rectangle (0.5,0.5);}}
\newcommand{\G}{\tikz{\useasboundingbox (0,0) rectangle (0.5,0.5);
                      \draw[fill=gray!50] (0,0) rectangle (0.5,0.5);}}
\newcommand{\N}{\tikz{\useasboundingbox (0,0) rectangle (0.5,0.5);}}

\title[The rank of the $5\times 5$ permanent tensor is 16]{The rank of the $5\times 5$ permanent tensor is sixteen}
\author{Jong In Han}
\address{Jong In Han, School of Mathematics, Korea Institute for Advanced Study (KIAS), 85 Hoegi-ro, Dongdaemun-gu, Seoul, 02455, Republic of Korea}
\email{jihan09@kias.re.kr}

\author{Jeyoung Song}
\address{Jeyoung Song, Department of Mathematics and Statistical Science, University of Idaho, 875 Perimeter Drive, MS 1103, Moscow, Idaho 83844-1103, United States of America}
\email{jeyoungsong@uidaho.edu}


\begin{document}
\begin{abstract}
	In this paper, we show that the tensor rank of the $5\times 5$ permanent tensor is exactly $16$ over fields of characteristic zero, by proving the lower bound matching the known upper bound. The $5\times 5$ permanent tensor is a symmetric tensor corresponding to the monomial $x_1x_2x_3x_4x_5$, whose Waring rank is known to be $16$. Previously, it was known, by the higher-order Koszul flattening and Fischer's formula, that its tensor rank is either $15$ or $16$.
\end{abstract}
\maketitle

\section{Introduction}\label{sec: introduction}
Let $\Bbbk$ be a field of characteristic zero and $V_i=\Bbbk^{n_i}$ for $i=1,\ldots,d$, and let $T\in V_1\otimes\cdots\otimes V_d$. The \emph{tensor rank} $\rk(T)$ is the smallest integer $r$ such that $T$ can be written as a sum of $r$ rank-one tensors. When $d=2$, this notion coincides with the usual rank of a matrix. Unlike matrix rank, which can be computed by elementary linear algebra, tensor rank is notoriously difficult to determine for tensors of order $d\geq 3$. For example, for the matrix multiplication tensor \[M_{\la n,m,k\ra}\coloneqq\sum_{u,v,w} a_{u,v}\otimes b_{v,w}\otimes c_{w,u}\in\Bbbk^{nm}\otimes \Bbbk^{mk}\otimes \Bbbk^{kn},\] where $a_{u,v}$, $b_{v,w}$, and $c_{w,u}$ denote the standard basis vectors of the three factors, the rank is known exactly only in the cases $(n,m,k)=(2,2,2)$, $(2,2,3)$, and $(2,2,4)$, up to permutation of $n,m,k$ \cite{MR248973,MR297115,MR780851,MR3457248}. Tensor rank is also central in algebraic complexity theory, since it often measures the number of arithmetic operations required by a corresponding algorithm. In particular, for $T\in V_1\otimes V_2\otimes V_3$, the tensor rank of $T$ is the minimum number of multiplications required by a bilinear algorithm, without exploiting commutativity, to compute the associated bilinear map $V_1^*\times V_2^*\to V_3$.

The determinant and permanent are a classical pair to compare in complexity theory, and separating them by a complexity measure is a long-standing problem, of which Valiant's conjecture \cite{MR564634} is the most famous instance.
Valiant's conjecture concerns them as polynomials of degree $d$ in $d^2$ variables, whereas the object of this paper is the permanent tensor in $(\Bbbk^d)^{\otimes d}$.

Let $V_1=\cdots=V_d=\Bbbk^d$ and let $e_1,\ldots,e_d$ be the standard basis of $\Bbbk^d$. The \emph{determinant} and \emph{permanent} tensors of order $d$ are defined by
\begin{align*}
\det_d
&=
\sum_{\sigma\in\mathfrak{S}_d}
\sgn(\sigma)e_{\sigma(1)}\otimes\cdots\otimes e_{\sigma(d)}
\in
V_1\otimes\cdots\otimes V_d,\\
\per_d
&=
\sum_{\sigma\in\mathfrak{S}_d}
e_{\sigma(1)}\otimes\cdots\otimes e_{\sigma(d)}
\in
V_1\otimes\cdots\otimes V_d
\end{align*}
where $\mathfrak{S}_d$ is the symmetric group on $d$ letters.
For every $d\ge 3$, the tensors $\det_d$ and $\per_d$ are separated by their tensor ranks: $\rk(\per_d)<\rk(\det_d)$; see \cite{HJK25} for details.

Note that $\per_d$ is a symmetric tensor.
Under the standard identification of symmetric tensors in $V_1\otimes\cdots\otimes V_d$ with degree-$d$ forms in $d$ variables, it corresponds to the squarefree monomial
\[
    x_1x_2\cdots x_d
\]
up to a nonzero scalar.

The best known upper bound for $\rk(\per_d)$ comes from an explicit decomposition:
\[
\per_d
=
\frac{1}{2^{d-1}}
\sum_{\substack{(\delta_1,\ldots,\delta_d)\in\{1,-1\}^{d}\\ \delta_1=1}}
\left(\prod_{k=1}^d \delta_k\right)
\left(\sum_{j=1}^d \delta_j e_j\right)
\otimes\cdots\otimes
\left(\sum_{j=1}^d \delta_j e_j\right).
\]
The identity has its roots in classical polarization formulas, with antecedents in Serret \cite{Ser1869} and Mazur--Orlicz \cite{MO34}, and a later explicit formulation in Fischer \cite{MR1573008}.
Note that this is a Waring decomposition.
Thus
\[
\rk(\per_d)\le \wrk(\per_d)\leq 2^{d-1}
\]
where $\wrk(T)$ denotes the \emph{Waring rank} (or the \emph{symmetric rank}) of a symmetric tensor $T$, i.e., the smallest integer $r$ such that $T$ can be written as a sum of $r$ rank-one symmetric tensors.
This decomposition is optimal among symmetric decompositions: the matching lower bound $\wrk(\per_d)\ge 2^{d-1}$ was shown by Ranestad and Schreyer \cite{MR2842085}, so that \[\wrk(\per_d)=2^{d-1}.\]
However, a rank decomposition need not consist of rank-one symmetric tensors, so this does not settle $\rk(\per_d)$ and the exact value was previously unknown for every $d\ge 5$.

The \emph{border rank} $\brk(T)$ of a tensor $T$ is the smallest integer $r$ such that $T$ lies in the Zariski closure of the set of tensors of rank at most $r$.
In 2016, Ilten and Teitler \cite{MR3492642} showed
\[
    \rk(\per_3)\ge 4,
\]
and in 2019, Derksen and Makam \cite{MR3987583} showed
\[
    \brk(\per_3)\ge 4.
\]
Together with Fischer's formula, these determine the rank and border rank of $\per_3$:
\[
    \brk(\per_3)=\rk(\per_3)=4.
\]
In 2016, Derksen \cite{MR3494510} showed
\[
\rk(\per_d)\ge \brk(\per_d)\ge \binom{d}{\lfloor d/2\rfloor},
\]
and in 2024, Houston, Goucher, and Johnston \cite{MR4822414} improved the lower bound for $\rk(\per_d)$ by one:
\[
\rk(\per_d)\ge \binom{d}{\lfloor d/2\rfloor}+1.
\]
To the best of our knowledge, no better lower bound is known for $d\ge 8$.
Later, we see the same bound holds for the border rank over $\c$ and its subfields (\Cref{prop:brk_large_d}).
For $d\le 7$, stronger bounds are known: in 2021, Krishna and Makam \cite{MR4284788} showed
\[
    \brk(\per_5)\ge 13\text{ and }\brk(\per_7)\ge 42,
\]
and recently, the first author, with Ju and Kim \cite{HJK25}, showed 
\[
    \brk(\per_4)\ge 8\text{, }\brk(\per_5)\ge 15\text{, }\brk(\per_6)\ge 29\text{, }\brk(\per_7)\ge 55.
\]
Note that this determines the rank and border rank for $d=4$: \[\brk(\per_4)=\rk(\per_4)=8.\]

The exact value of $\rk(\per_5)$, previously known to be either 15 or 16, is the subject of this paper: we show $\rk(\per_5)=16$ over fields of characteristic zero.
The proof requires computing the locus where a $2500\times 2500$ matrix with variables has rank at most 1344, which is infeasible to compute via the ideal of minors since there are $\binom{2500}{1345}^2$ minors.
We overcome this by using Conner--Harper--Landsberg's method described in \Cref{sec: preliminaries}.

\begin{thmalphabetintro}\label{thmintro: rk perm5 is 16}
	It holds that $\rk(\per_5)=16$ over fields of characteristic zero.
\end{thmalphabetintro}

This answers \cite[Question 5.8]{HJK25} for $d=5$, which asks whether $\rk(\per_d)=2^{d-1}$ for each $d\ge 5$.

\section*{Acknowledgement}
J. I. Han was supported by a KIAS Individual Grant (MG101401) at Korea Institute for Advanced Study.

\section{Conner--Harper--Landsberg's method on the rank locus}\label{sec: preliminaries}
In \cite{MR4595287}, Conner, Harper, and Landsberg introduced an algorithm to find the locus $X\subseteq \a_\Bbbk^r$ of the points $(x_1,\dots,x_r)\in \a_\Bbbk^r$ such that $\rank M(x_1,\dots,x_r)\le k$ where $M\in\Mat_{m\times n}(\Bbbk[x_1,\ldots,x_r])$ and $\Bbbk$ is an algebraically closed field.
Although the ideal of $(k+1)\times (k+1)$-minors of $M$ defines the same locus, computing the ideal of minors becomes infeasible very fast as $m,n,k$ increase.

Let $M\in\Mat_{m\times n}(R)$ where $R=(\Bbbk[x_1,\dots,x_r]/I)_g$ for some nonzero $g\in \Bbbk[x_1,\dots,x_r]/I$.
We want to find the locus $X(M,I,g,k)$ of points $(x_1,\dots,x_r)\in V(I)\cap D(g)$ such that $\rank(M(x_1,\dots,x_r))\le k$ by a recursive algorithm.
Note that since $D(g)$ is open, the locus $X(M,I,g,k)$ need not be closed.
The algorithm returns an ideal $J$ such that
\begin{equation}\label{eqn: V(J)}
    V(J)=\overline{X(M,I,g,k)}.
\end{equation}
The closure does not harm the recursion as closure commutes with finite unions.

Suppose $M$ has a unit as an entry, i.e., $M_{i,j}$ is invertible in $R$ for some $i,j$.
In that case, we perform row and column operations to make all other entries on the $i$-th row and $j$-th column zero.
Then we remove the $i$-th row and $j$-th column and denote the $(m-1)\times(n-1)$ matrix by $M'$.
Now, the locus $X(M,I,g,k)$ is the same as the set $X(M',I,g,k-1)$.
Hence this reduces the problem and we repeat this as much as possible.

However, there may not be any unit in the matrix at some step.
In that case, we pick a nonzero entry $M_{i,j}=:f$ and divide the cases to use the principle
\begin{equation}\label{eqn: principle}
    X(M,I,g,k)=(X(M,I,g,k)\cap V(f))\cup (X(M,I,g,k)\cap D(f)).
\end{equation}
In other words, we determine the locus where $f=0$ at $(x_1,\dots,x_r)$ and the locus where $f\ne 0$ at $(x_1,\dots,x_r)$ separately.

When $f=0$ at $(x_1,\dots,x_r)$, we need to determine the locus of points $(x_1,\dots,x_r)\in V(I+(f))\cap D(g)$ at which $\rank M_1(x_1,\dots,x_r)\le k$ where $M_1\in\Mat_{m\times n}(R/(f))$ is the matrix $M$ considered as a matrix over the ring $R/(f)$.
This case reduces to finding $X(M_1,I+(f),g,k)$.

When $f\ne 0$ at $(x_1,\dots,x_r)$, we can regard $f$ as a unit.
Using $f$ as a pivot, we eliminate other entries in the same row and column as $f$.
Then we remove that row and column and denote the resulting $(m-1)\times (n-1)$ matrix by $M_2$.
We consider $M_2$ as a matrix over the ring $R_f$, i.e., $M_2\in \Mat_{(m-1)\times(n-1)}(R_f)$. 
Then it is equivalent to determine the locus of points $(x_1,\dots,x_r)\in V(I)\cap D(gf)$ at which $\rank M_2(x_1,\dots,x_r)\le k-1$.
This case reduces to finding $X(M_2,I,gf,k-1)$.

Note that in both cases, the problem reduces: in the case $f=0$, the number of zero entries increases, while in the case $f\ne 0$, all of $m,n,k$ decrease by one.
The recursion terminates at the following cases.
\begin{itemize}
    \item $X(M,I,g,k)=V(I)\cap D(g)$ when $M=0$ and $k\ge 0$;
    \item $X(M,I,g,k)=V(I+(M_{i,j})_{i,j})\cap D(g)$ when $k=0$;
    \item $X(M,I,g,k)=\varnothing$ when $k<0$.
\end{itemize}
The algorithm terminates after finitely many steps.

In the implementation, the code remembers the ideal of relations $I$.
We consider $I$ as an ideal in $\Bbbk[x_1,\dots,x_r,u]$ where $u$ is an additional variable to make some entries units.
Initially, the ideal $I$ is set to the zero ideal.
When it branches out to the case $f=0$, we add $f$ to the relations $I$ as $f$ must be zero.
When it branches out to the case $f\ne 0$, we add $uf-1$ to the relations $I$ to make $f$ a unit.
Note that if it branches out again to the cases $f'=0$ and $f'\ne 0$, then there is the case that $f$ and $f'$ are both units.
To not introduce a new variable, we eliminate $u$ by replacing $I$ with $I\cap\Bbbk[x_1,\dots,x_r]$ and add the relation $uff'-1$ to $I$.
This elimination is the source of the closure in \eqref{eqn: V(J)}.
By remembering the entries that must be units, we do not need to introduce a new variable other than $u$.

\section{Proof of \Cref{thmintro: rk perm5 is 16}}\label{sec: proof of theorem}
In \cite{MR3987862}, Hauenstein, Oeding, Ottaviani, and Sommese introduced the following higher-order Koszul flattening, which applies the Koszul flattening, introduced by Landsberg and Ottaviani \cite{MR3376667}, to higher-order tensors.

\begin{lemma}[higher-order Koszul flattening, {\cite{MR3376667,MR3987862}}]\label{lem:higher order koszul}
Let $T=\sum_{l=1}^{k}v_{l,1}\otimes \cdots\otimes v_{l,d}\in V_{1}\otimes \cdots\otimes V_{d}$ where $d\ge 3$ and $\dim V_{i}=n_{i}$ for $1\leq i\leq d$.
Fix integers $p_1,\ldots,p_{d-2}$ such that $0\le p_i\le n_i-1$ for all $i=1,\ldots,d-2$. Define $$\begin{aligned}
T^{\wedge(p_{1},\ldots,p_{d-2})}_{(V_{1},\ldots,V_{d-2})}: \left(\bigotimes_{i=1}^{d-2}\bigwedge^{p_{i}}V_{i}\right)\otimes V_{d}^{*}\to\left(\bigotimes_{i=1}^{d-2}\bigwedge^{p_{i}+1}V_{i}\right)\otimes V_{d-1}
\end{aligned}$$
by $$\begin{aligned}
&\left(\bigotimes_{i=1}^{d-2}a_{i,1}\wedge \cdots\wedge a_{i,p_{i}}\right)\otimes f\\
&\mapsto \sum_{l=1}^{k}f(v_{l,d}) \left(\bigotimes_{i=1}^{d-2}v_{l,i}\wedge a_{i,1}\wedge \cdots\wedge a_{i,p_{i}} \right)\otimes v_{l,d-1}
\end{aligned}$$
and extending linearly.
Then $$\begin{aligned}
\left\lceil \frac{\rank(T^{\wedge(p_{1},\ldots,p_{d-2})}_{(V_{1},\ldots,V_{d-2})})}{\binom{n_{1}-1}{p_{1}}\cdots \binom{n_{d-2}-1}{p_{d-2}}}\right\rceil\leq \brk(T).
\end{aligned}$$
\end{lemma}
\begin{proof}
Define a linear map 
\begin{align*}
    \Phi:V_{1}\otimes \cdots\otimes V_{d} &\to \hom(\Bbbk^{N_{1}},\Bbbk^{N_{2}})\\
    \mathcal{T} &\mapsto \mathcal{T}^{\wedge(p_{1},\ldots,p_{d-2})}_{(V_{1},\ldots,V_{d-2})}
\end{align*}
where $N_{1}=\binom{n_{1}}{p_{1}}\cdots \binom{n_{d-2}}{p_{d-2}}n_{d}$ and $N_{2}=\binom{n_{1}}{p_{1}+1}\cdots \binom{n_{d-2}}{p_{d-2}+1}n_{d-1}$.
Note that $\Phi$ is well defined since $\mathcal{T}^{\wedge(p_{1},\ldots,p_{d-2})}_{(V_{1},\ldots,V_{d-2})}$ does not depend on the decomposition of $\mathcal{T}$.
For any nonzero simple tensor $S\in V_{1}\otimes \cdots\otimes V_{d}$, we have $\rank(\Phi(S))=\binom{n_{1}-1}{p_{1}}\cdots \binom{n_{d-2}-1}{p_{d-2}}\eqqcolon q$.
Thus
\begin{align*}
    \brk(T)
    &\geq \brk_{\Phi(\Seg(\mathbb{P}V_{1}\times \cdots\times \mathbb{P}V_{d}))}(\Phi(T))\\
    &\geq \brk_{\sigma_{q}(\Seg(\mathbb{P}_{\Bbbk}^{N_{1}-1}\times \mathbb{P}_{\Bbbk}^{N_{2}-1}))}(\Phi(T))\\
    &=\left\lceil \frac{\rank(\Phi(T))}{q}\right\rceil
\end{align*}
gives the result where $\sigma_r(X)$ denotes the $r$-secant variety of $X$ and $\brk_X(p)$ denotes the minimum integer $r$ such that $p\in\sigma_r(X)$.
\end{proof}

We often denote the higher-order Koszul flattening of $T$ simply by $T^{\wedge(p_{1},\ldots,p_{d-2})}$ when the choice of ${(V_{1},\ldots,V_{d-2})}$ is clear from the context.
If $n_{1}=\cdots=n_{d}=d$ and $p_{i}=i$ for $i=1,\ldots,d-2$, the higher-order Koszul flattening $T^{\wedge(p_{1},\ldots,p_{d-2})}$ is a linear map, which is represented by a square matrix of size $\prod_{i=1}^{d-1}\binom{d}{i}$. Note that, for the determinant tensor $\det_{d}$ of any order $d$, it is known that the higher-order Koszul flattening $(\det_{d})^{\wedge(1,\ldots,d-2)}$ with $(p_1,\dots,p_{d-2})=(1,\dots,d-2)$ has full rank (\cite[Theorem 4.2]{HJK25}) while that of $\per_d$ does not. In practice, we may expect a better lower bound when the induced map is represented by a square matrix since the matrix rank is not restricted by the smaller side. We use the numbers $p_{i}=i$ for $i=1,2,3$ for the higher-order Koszul flattening of $\per_{5}$ as well.

The higher-order Koszul flattening gives us only $\brk(\per_5)\ge 15$ as
\[
    \rank\left((\per_5)^{\wedge(1,2,3)}\right)=1426,
\]
and $\left\lceil\frac{1426}{96}\right\rceil=15$.
To obtain $\rk(\per_5)\ge 16$, we show that $\brk(\per_5-S)\ge 15$ for every nonzero simple tensor \[S=\left(\sum_{j=1}^{5}s_{1,j}e_{j}\right)\otimes\cdots\otimes\left(\sum_{j=1}^{5}s_{5,j}e_{j}\right).\]
Using the symmetry, we divide into cases and reduce the number of parameters as much as possible.
Even so, in each case we have to determine the locus of points $(s_{i,j})_{i,j}$ at which a $2500\times 2500$ matrix with entries in $\Bbbk[(s_{i,j})_{i,j}]$ has rank at most $1344$.
Since computing these loci via the ideal of minors is infeasible, we instead use Conner--Harper--Landsberg's method, which enables us to show that every such locus is empty.

Now we prove the main theorem.

\begin{thmalphabetmaintext}\label{thm: rk perm5 is 16}
	It holds that $\rk(\per_5)=16$ over fields of characteristic zero.
\end{thmalphabetmaintext}
\begin{proof}
We may assume $\Bbbk$ is algebraically closed.
It is enough to show
\[
    \rk(\per_{5}-S)\ge 15
\]
for every nonzero simple tensor $S=s_{1}\otimes s_{2}\otimes s_{3}\otimes s_{4} \otimes s_{5}$ where $s_{i}=\sum_{j=1}^{5}s_{i,j}e_{j}\in V_{i}$ for each $i=1,\ldots,5$.
We consider the $s_{i,j}$ as variables.
We denote by $M$ the $2500\times 2500$ matrix corresponding to the higher-order Koszul flattening
$$
\begin{aligned}
(\per_{5}-S)^{\wedge(1,2,3)}_{(V_1,V_2,V_3)}: V_{1}\otimes\wedge^{2} V_{2}\otimes\wedge^{3} V_{3}\otimes V_{5}^{*} \to \wedge^{2} V_{1}\otimes\wedge^{3} V_{2}\otimes\wedge^{4} V_{3}\otimes V_{4}
\end{aligned}$$
of $\per_{5}-S\in V_1\otimes V_2\otimes V_3\otimes V_4\otimes V_5$ with respect to $p_1=1$, $p_2=2$, and $p_3=3$.

Let
\[
    G=\left[\Big((\Bbbk^{*})^{4}\rtimes\mathfrak{S}_{5}\Big)\cdot (\Bbbk^{*})^4\right]\rtimes \mathfrak{S}_{5}
\]
be a group acting on $V_{1}\otimes \cdots\otimes V_{5}$ where $A\cdot B$ denotes the subgroup of $\GL(V_1)\times \cdots \times \GL(V_5)$ generated by $A$ and $B$.
It acts on the standard basis as follows.
\begin{enumerate}[label=(\alph*)]
    \item\label{item: a} The first $(\Bbbk^{*})^{4}$:
    \[
        (a_{1},\ldots,a_{4})\cdot(e_{i_{1}}\otimes \cdots\otimes e_{i_{5}})=a_{i_{1}}e_{i_{1}}\otimes \cdots\otimes a_{i_{5}}e_{i_{5}}
    \]
    where $a_5=(a_1a_2a_3a_4)^{-1}$.
    \item\label{item: b} The first $\mathfrak{S}_{5}$:
    \[
        \sigma\cdot(e_{i_{1}}\otimes \cdots \otimes e_{i_{5}})=e_{\sigma(i_{1})}\otimes \cdots\otimes e_{\sigma(i_{5})}.
    \]
    \item\label{item: c} The second $(\Bbbk^{*})^{4}$:
    \[
        (a_{1},\ldots,a_{4})\cdot(e_{i_{1}}\otimes \cdots\otimes e_{i_{5}})=a_{1}e_{i_{1}}\otimes \cdots\otimes a_{5}e_{i_{5}}
    \]
    where $a_5=(a_1a_2a_3a_4)^{-1}$.
    \item\label{item: d} The second $\mathfrak{S}_{5}$:
    \[
        \sigma\cdot(e_{i_{1}}\otimes \cdots \otimes e_{i_{5}})=e_{i_{\sigma^{-1}(1)}}\otimes \cdots\otimes e_{i_{\sigma^{-1}(5)}}.
    \]
\end{enumerate}
Each action extends linearly to $V_{1}\otimes \cdots\otimes V_{5}$.

Usually, the \emph{symmetry group of a tensor} $T\in V_1\otimes \cdots \otimes V_n$ is defined as the stabilizer of $T$ in $\GL(V_1)\times \cdots \times \GL(V_n)/(\Bbbk^*)^{n-1}$ where
\[
    (\Bbbk^*)^{n-1}=\{(\lambda_1\mathrm{Id}_{V_1},\dots,\lambda_n\mathrm{Id}_{V_n})\mid \prod_{i=1}^{n}\lambda_i=1\}
\]
is the kernel of $\GL(V_1)\times \cdots \times \GL(V_n)\to \GL(V_1\otimes\cdots\otimes V_n)$ (see, e.g., \cite{MR4595287}).
However, we do not take the quotient of the group by $(\Bbbk^*)^{n-1}$: although $(\Bbbk^*)^{n-1}$ acts trivially on $V_1\otimes \cdots \otimes V_n$, it acts nontrivially on the coefficients of the simple tensor $S$, and we exploit this also to reduce the number of parameters representing $S$.
One can easily see
\[
    \Big((\Bbbk^{*})^4\rtimes\mathfrak{S}_5\Big)\cdot (\Bbbk^{*})^4\subseteq \mathrm{Stab}_{\GL(V_1)\times \cdots \times \GL(V_5)}(\per_5)
\]
where $\mathrm{Stab}$ denotes the stabilizer and the second $(\Bbbk^*)^4$ denotes the kernel described above.
We also utilize the group action \ref{item: d} by $\mathfrak{S}_{5}$ permuting the factors to further reduce the number of cases.
Then it holds that
\[
    G=\left[\Big((\Bbbk^{*})^{4}\rtimes\mathfrak{S}_{5}\Big)\cdot (\Bbbk^{*})^4\right]\rtimes \mathfrak{S}_{5}\subseteq \mathrm{Stab}_{(\GL(V_1)\times \cdots \times \GL(V_5))\rtimes \mathfrak{S}_5}(\per_5).
\]
Each of the four actions \ref{item: a}--\ref{item: d} is used in the case analysis below.

Note that the $((\GL(V_1)\times \cdots \times \GL(V_5))\rtimes \mathfrak{S}_5)$-action preserves the tensor rank and the border rank, and hence so does the $G$-action.
Thus it holds that
\[
    \brk(\per_5-S)=\brk(g\cdot(\per_5-S))=\brk(\per_5-g\cdot S)
\]
for every $g\in G$.
Hence it is enough to show $\brk(\per_5-S)\ge 15$ for one $S$ in each $G$-orbit.

As $S\ne 0$, it holds that for each $i$, the coefficient $s_{i,j}$ is nonzero for some $j$.
Hence we may choose five nonzero coefficients, one in each column of the transposed matrix $((s_{i,j})_{1\le i,j\le 5})^{\mathrm{T}}$.
We consider cases according to the positions of these five nonzero coefficients: each partition of 5 records how many of these five coefficients lie in each row.
For example, the partition $(2,2,1)$ corresponds to the case that there are two of them in some row, two of them in another row, and one in a third row.
There are seven cases, corresponding to the seven partitions:
\begin{align*}
    &(5),\ (4,1),\ (3,2),\ (3,1,1), \\
    &(2,2,1),\ (2,1,1,1),\ (1,1,1,1,1).
\end{align*}
In the diagrams below, the box in the $i$-th column and $j$-th row represents $s_{i,j}$.
A gray box means the entry is normalized to 1, a white box means the entry is assumed to be 0.
Boxes not shown are free parameters.
\\
\\
\noindent
\textbf{Case $(5)$.}
There exists $j\in\{1,2,3,4,5\}$ such that $s_{1,j},s_{2,j},s_{3,j},s_{4,j}$, and $s_{5,j}$ are nonzero.
Using the first $\mathfrak{S}_5$-action \ref{item: b}, we may assume $j=1$.
Using the first $(\Bbbk^*)^4$-action \ref{item: a}, we may assume $s_{1,1}s_{2,1}s_{3,1}s_{4,1}s_{5,1}=1$.
Using the second $(\Bbbk^*)^4$-action \ref{item: c}, we may assume $s_{1,1}=s_{2,1}=s_{3,1}=s_{4,1}=s_{5,1}=1$.
This case has 20 free parameters.
\[
    \setlength{\arraycolsep}{0pt}
    \renewcommand{\arraystretch}{0}
    \begin{array}{ccccc}
    \G & \G & \G & \G & \G
    \end{array}
\]
\\
\noindent
\textbf{Case $(4,1)$.}
There exist $j_1,j_2\in\{1,2,3,4,5\}$ and a permutation $(i_1,\dots,i_5)$ of $(1,2,3,4,5)$ such that $s_{i_1,j_1},s_{i_2,j_1},s_{i_3,j_1},s_{i_4,j_1}$, and $s_{i_5,j_2}$ are nonzero.
Using the first $\mathfrak{S}_5$-actions \ref{item: b}, we may assume $j_1=1$ and $j_2=2$.
Using the second $\mathfrak{S}_5$-action \ref{item: d}, we may assume $i_1=1$, $i_2=2$, $i_3=3$, $i_4=4$, $i_5=5$.
Using the first $(\Bbbk^*)^4$-actions \ref{item: a}, we may assume $s_{1,1}s_{2,1}s_{3,1}s_{4,1}=1$ and $s_{5,2}=1$.
Using the second $(\Bbbk^*)^4$-action \ref{item: c}, we may assume $s_{1,1}=s_{2,1}=s_{3,1}=s_{4,1}=1$.
Then we can also assume $s_{5,1}=0$ since if not, it reduces to Case $(5)$.
This case has 19 free parameters.
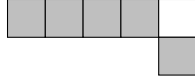
\[
    \setlength{\arraycolsep}{0pt}
    \renewcommand{\arraystretch}{0}
    \begin{array}{ccccc}
        \G & \G & \G & \G & \W \\
        \N & \N & \N & \N & \G \\
    \end{array}
\]
\\
\noindent
\textbf{Case $(3,2)$.}
As before, we may assume $s_{1,1}=s_{2,1}=s_{3,1}=s_{4,2}=s_{5,2}=1$ using $G$-actions.
Then we can also assume $s_{4,1}=s_{5,1}=0$ since if not, it reduces to Case $(4,1)$.
This case has 18 free parameters.
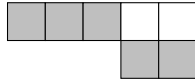
\[
    \setlength{\arraycolsep}{0pt}
    \renewcommand{\arraystretch}{0}
    \begin{array}{ccccc}
        \G & \G & \G & \W & \W \\
        \N & \N & \N & \G & \G \\
    \end{array}
\]
\\
\noindent
\textbf{Case $(3,1,1)$.}
As before, we may assume $s_{1,1}=s_{2,1}=s_{3,1}=s_{4,2}=s_{5,3}=1$ using $G$-actions.
We can also assume $s_{4,1}=s_{5,1}=0$ since if not, it reduces to Case $(4,1)$, and $s_{4,3}=s_{5,2}=0$ since if not, it reduces to Case $(3,2)$.
This case has 16 free parameters.
\[
    \setlength{\arraycolsep}{0pt}
    \renewcommand{\arraystretch}{0}
    \begin{array}{ccccc}
        \G & \G & \G & \W & \W \\
        \N & \N & \N & \G & \W \\
        \N & \N & \N & \W & \G \\
    \end{array}
\]
\\
\noindent
\textbf{Case $(2,2,1)$.}
As before, we may assume $s_{1,1}=s_{2,1}=s_{3,2}=s_{4,2}=s_{5,3}=1$ using $G$-actions.
We can also assume $s_{5,1}=s_{5,2}=0$ by Case $(3,2)$, and $s_{1,2}=s_{2,2}=s_{3,1}=s_{4,1}=0$ by Case $(3,1,1)$.
This case has 14 free parameters.
\[
    \setlength{\arraycolsep}{0pt}
    \renewcommand{\arraystretch}{0}
    \begin{array}{ccccc}
        \G & \G & \W & \W & \W \\
        \W & \W & \G & \G & \W \\
        \N & \N & \N & \N & \G \\
    \end{array}
\]
\\
\noindent
\textbf{Case $(2,1,1,1)$.}
As before, we may assume $s_{1,1}=s_{2,1}=s_{3,2}=s_{4,3}=s_{5,4}=1$ using $G$-actions.
We can also assume $s_{3,1}=s_{4,1}=s_{5,1}=0$ by Case $(3,1,1)$, and $s_{3,3}=s_{3,4}=s_{4,2}=s_{4,4}=s_{5,2}=s_{5,3}=0$ by Case $(2,2,1)$.
This case has 11 free parameters.
\[
    \setlength{\arraycolsep}{0pt}
    \renewcommand{\arraystretch}{0}
    \begin{array}{ccccc}
        \G & \G & \W & \W & \W \\
        \N & \N & \G & \W & \W \\
        \N & \N & \W & \G & \W \\
        \N & \N & \W & \W & \G \\
    \end{array}
\]
\\
\noindent
\textbf{Case $(1,1,1,1,1)$.}
As before, we may assume $s_{1,1},s_{2,2},s_{3,3},s_{4,4}$, and $s_{5,5}$ are nonzero using the first and second $\mathfrak{S}_5$-actions.
Using the first $(\Bbbk^*)^4$-actions \ref{item: a}, we may assume $s_{1,1}=s_{2,2}=s_{3,3}=s_{4,4}=1$.
In this case, we cannot assume $s_{5,5}=1$ even if it is nonzero, unlike the other cases.
This is because neither of the $(\Bbbk^*)^4$-actions \ref{item: a} and \ref{item: c} changes $s_{1,1}s_{2,2}s_{3,3}s_{4,4}s_{5,5}$.
As before, we can assume $s_{i,j}=0$ for any $i\neq j$ since if not, it reduces to the previous cases.
This case has one free parameter $s_{5,5}$.
\[
    \setlength{\arraycolsep}{0pt}
    \renewcommand{\arraystretch}{0}
    \begin{array}{ccccc}
        \G & \W & \W & \W & \W \\
        \W & \G & \W & \W & \W \\
        \W & \W & \G & \W & \W \\
        \W & \W & \W & \G & \W \\
        \W & \W & \W & \W & \N \\
    \end{array}
\]
\\
\\
\indent
By \Cref{lem:higher order koszul}, it is enough to show that
\[
    \rank(M(s_{1,1},s_{1,2},\dots,s_{5,5}))>14\binom{4}{1}\binom{4}{2}\binom{4}{3}=1344
\]
for every $(s_{1,1},s_{1,2},\dots,s_{5,5})$ in each of the seven cases above.

For each case, we apply Conner--Harper--Landsberg's method to compute the locus of $\rank M((s_{i,j})_{i,j})\le 1344$ using Macaulay2 \cite{M2}.
The code can be found at \url{https://github.com/jihan099/rkperm5}.
In the code, the computation is done over $\q$.
This suffices for our purpose: since the matrix $M$ has entries in $\q[(s_{i,j})_{i,j}]$ and every step of the algorithm involves only arithmetic operations with rational coefficients, one may regard the computation as being carried out over $\Bbbk$, where it happens that all the coefficients involved lie in $\q$.
The code shows that the resulting ideal contains 1 in each case, so each locus is empty over $\Bbbk$.
Therefore we get
\[
    \rank\left((\per_{5}-S)^{\wedge(1,2,3)}_{(V_1,V_2,V_3)}\right)\ge 1345,
\]
which implies, by \Cref{lem:higher order koszul}, that $\brk(\per_{5}-S)\geq 15$ for every $S$ in the above seven cases, hence for every nonzero simple tensor $S$.
This shows $\rk(\per_5)\ge 16$.
Since it is known that $\rk(\per_5)\le 16$ over $\q$ by Fischer's formula, we conclude that $\rk(\per_5)=16$ over $\Bbbk$.
\end{proof}

\section{On lower bounds for $\brk(\per_d)$}
For $d\ge 8$, the best known lower bound for $\rk(\per_d)$ is $\binom{d}{\lfloor \frac{d}{2}\rfloor}+1$ (\cite{MR4822414}) while that for $\brk(\per_d)$ is $\binom{d}{\lfloor\frac{d}{2}\rfloor}$ (\cite{MR3494510}).
Here we remark that the lower bound for $\brk(\per_d)$ can be improved to $\binom{d}{\lfloor \frac{d}{2}\rfloor}+1$ over $\c$ and its subfields using a well-known argument.

\begin{proposition}\label{prop:brk_large_d}
It holds that $$\brk(\per_{d})\geq \binom{d}{\lfloor d/2\rfloor}+1$$
over $\c$ and its subfields for $d\ge 3$.
\end{proposition}
\begin{proof}
It is enough to prove the statement over $\c$.
Let $T\coloneqq \per_{d}$ and $\{i_{1},\ldots,i_{a}\}\sqcup\{j_{1},\ldots,j_{b}\}$ be a partition of $\{1,\ldots,d\}$ with $a\ge 1$, $b\ge 2$, and $b=d-a$. Suppose that $\brk(T)=\rank(f_{i_{1},\ldots,i_{a}}(T))$ where $$f_{i_{1},\ldots,i_{a}}:V_{1}\otimes \cdots\otimes V_{d}\to \hom(V_{i_{1}}^{*}\otimes \cdots\otimes V_{i_{a}}^{*},V_{j_{1}}\otimes \cdots\otimes V_{j_{b}})$$ is the canonical isomorphism.
Then $\brk(T)=\binom{d}{a}$ so that $T=\lim_{ k \to \infty }T_{k}$ for some $T_{k}\in V_{1}\otimes \cdots\otimes V_{d}$ with $\rk(T_{k})=\binom{d}{a}$. Let $L_{k}=\Im(f_{i_{1},\ldots,i_{a}}(T_{k}))$ and $L=\Im(f_{i_{1},\ldots,i_{a}}(T))$. Then there are $\binom{d}{a}$ independent simple tensors $w_{k,1},\ldots,w_{k,\binom{d}{a}}\in V_{j_1}\otimes\cdots\otimes V_{j_b}$ for each $k$ such that $$L_{k}\subseteq\left\langle w_{k,1},\ldots,w_{k,\binom{d}{a}}\right\rangle\eqqcolon W_{k}.$$ Indeed, we may take the parts in $V_{j_1}\otimes\cdots\otimes V_{j_b}$ of the simple tensors in the decomposition of $T_k$. If those are not linearly independent, we keep a maximal linearly independent subset of them and replace the remaining ones by any other simple tensors in $V_{j_1}\otimes\cdots\otimes V_{j_b}$ that make the entire set linearly independent. By the compactness of $\Gr\left(\binom{d}{a},V_{j_{1}}\otimes \cdots\otimes V_{j_{b}}\right)$, there is a subsequence $$W_{k_{1}},W_{k_{2}},W_{k_{3}},\ldots$$ converging to some $W\in \Gr\left(\binom{d}{a},V_{j_{1}}\otimes \cdots\otimes V_{j_{b}}\right)$.
After reindexing, we may assume $\{W_k\}_{k\in\mathbb{N}}$ converges to $W$. As $f_{i_{1},\ldots,i_{a}}(T)(v)=\lim_{ k \to \infty }f_{i_{1},\ldots,i_{a}}(T_k)(v)$ for any $v\in V_{i_{1}}^*\otimes \cdots\otimes V_{i_{a}}^*$ and $f_{i_{1},\ldots,i_{a}}(T_k)(v)\in W_{k}$ for each $k\geq1$, we have $f_{i_{1},\ldots,i_{a}}(T)(v)\in W$ which means $L\subseteq W$. Since $\dim(W)=\binom{d}{a}=\dim(L)$, we get $L=W$.

By the compactness of $\Seg(\p V_{j_{1}}\times \cdots\times \mathbb{P}V_{j_{b}})$, there is a subsequence of $\{w_{k,1}\}_{k\in\mathbb{N}}$ such that $$[w_{k_1,1}],[w_{k_2,1}],[w_{k_3,1}],\ldots$$ converges to $[w]$ for some $[w]\in \Seg(\mathbb{P}V_{j_{1}}\times \cdots\times \mathbb{P}V_{j_{b}})$ where $[\cdot]$ denotes the canonical projection. 
After reindexing, we may assume $\{[w_{k,1}]\}_{k\in\mathbb{N}}$ converges to $[w]$.
Since $w_{k,1}\in W_k$ for each $k\geq1$, we have $w\in W$. Hence $w\in L$. This is a contradiction since $$L=\left\langle\sum_{\sigma\in \mathfrak{S}_{b}}^{} e_{l_{\sigma(1)}}\otimes \cdots\otimes e_{l_{\sigma(b)}}\mid 1\leq l_{1}<\cdots<l_{b}\leq d \right\rangle$$ does not contain a nonzero simple tensor when $b\geq 2$. This is because $L\subseteq \Sym^{b}(\c^d)$ after identifying $V_{j_1},\ldots,V_{j_b}$ with $\mathbb C^d$ and a nonzero symmetric simple tensor must contain a nonzero coefficient for $(e_i)^{\otimes b}$ for some $i$ when one represents it as a linear combination of the standard basis $\{e_{l_1}\otimes\cdots\otimes e_{l_b}\mid 1\le l_1,\ldots,l_b\le d\}$ whereas $L$ does not. Hence $\brk(T)>\rank(f_{i_{1},\ldots,i_{a}}(T))=\binom{d}{a}$ when $1\leq a\leq d-2$. In particular, when $a=\lfloor d/2\rfloor$, we get the result.
\end{proof}

Note that while the tensor rank of $\per_5$ is settled in \Cref{thm: rk perm5 is 16}, the border rank is not: it remains open whether $\brk(\per_5)$ is $15$ or $16$.
In view of the pattern for $\per_3$ and $\per_4$, where the border rank matches the tensor rank, it is natural to expect that the same holds for $\per_5$.
We therefore pose the following question.

\begin{question}
    Does $\brk(\per_5)=16$ hold over $\c$?
    More generally, does $\brk(\per_d)=2^{d-1}$ hold for all $d\ge 5$?
\end{question}

\bibliographystyle{amsalpha}
\bibliography{ref.bib}
\end{document}